\documentclass{amsart}

\DeclareMathSymbol{\twoheadrightarrow}  {\mathrel}{AMSa}{"10}
        \def\GG{{\mathcal G}}

\def\Q{{\mathbb Q}}
\def\Z{{\mathbb Z}}
\def\C{{\mathbb C}}

\def\W{{\mathbb W}}

\def\D{{\mathbb D}}

\def\RR{{\mathbb R}}
\def\F{{\mathbb F}}
\def\P{{\mathbb P}}

\def\f{{\tilde F}}
                     \def\f0{{\mathfrak f}}

\def\A8{{\mathbf A}_8}

\def\RR{{\mathfrak R}}
\def\Perm{\mathrm{Perm}}
\def\Gal{\mathrm{Gal}}

\def\End{\mathrm{End}}
\def\Aut{\mathrm{Aut}}

                  \def\cl{\mathrm{cl}}

\def\ST{{\mathbf S}}

\def\A{\mathbf{A}}

\def\dim{\mathrm{dim}}
\def\Oc{{\mathcal O}}

\newtheorem{thm}{Theorem}[section]
\newtheorem{lem}[thm]{Lemma}

\theoremstyle{definition}
\newtheorem{defn}[thm]{Definition}
\newtheorem{ex}[thm]{Example}
\newtheorem{rem}[thm]{Remark}

\newtheorem{rems}[thm]{Remarks}
        \newtheorem{sect}[thm]{}
\title[Absolutely simple Prymians]{Absolutely simple Prymians of trigonal curves}

\author[Yuri G.\ Zarhin]{Yuri G.\ Zarhin}
\address{Department of Mathematics, Pennsylvania State University,
University Park, PA 16802, USA}
\address{Institute for Mathematical Problems in Biology, Russian Academy of
Sciences, Pushchino, Moscow Region, Russia}

\email{zarhin\char`\@math.psu.edu}

\dedicatory{In Memoriam of Vasilii Alekseevich Iskovskikh}

\begin{document}
\maketitle

\section{Introduction}
As usual, $\Z,\Q,\C$ denote the ring of integers, the field of rational numbers
and the field of complex numbers respectively.  Let us fix a primitive cube
root of unity $\zeta_3=\frac{-1+\sqrt{-3}}{2}\in \C$. Let
$\Q(\zeta_3)=\Q(\sqrt{-3})$ be the third cyclotomic field and
$\Z[\zeta_3]=\Z+\Z\cdot \zeta_3$ its ring of integers. We write $\lambda$ for
the (principal) maximal ideal $(1-\zeta_3)\cdot\Z[\zeta_3]$ of $\Z[\zeta_3]$.

It is known \cite[Th. 5 on p. 176]{Shimura2} (see also \cite{Oort}) that for all positive integers $m$
different from $2$ there exists a $m$-dimensional complex abelian variety, whose
endomorphism ring is $\Z[\zeta_3]$. Shimura's proof is purely complex-analytic
 and not constructive; roughly speaking it deals with points of the corresponding
 moduli space that do not belong to a countable union of subvarieties of positive
 codimension. In this paper we discuss a geometric approach to an explicit construction
 of those abelian varieties via jacobians, prymians and Galois theory.

In order to explain our approach, let us start with the following definitions.
Let $f(x)\in\C[x]$ be a polynomial of degree $n\ge 4$ without multiple roots.
Let $C_{f,3}$ be a smooth projective model of the smooth affine curve
$y^3=f(x)$. It is well known (\cite{Koo}, pp. 401-402, \cite{Towse}, Prop. 1 on
p. 3359, \cite{Poonen}, p. 148) that the genus $g(C_{f,3})$  of $C_{f,3}$ is
$n-1$ if $3$ does not divide $n$ and $n-2$ if it does. In both cases
$g(C_{f,3})\ge 3$ is {\sl not} congruent to $2$ modulo $3$.

The map $(x,y)\mapsto (x,\zeta_3 y)$ gives rise to a non-trivial birational
automorphism $\delta_3:C_{f,3}\to C_{f,3}$ of period $3$. By functoriality,
$\delta_3$ induces the linear operator in the space of differentials of the
first kind
$$\delta_3^{*}: \Omega^1(C_{f,3}) \to \Omega^1(C_{f,3}).$$
Its spectrum consists of eigenvalues $\zeta_3^{-1}$ and $\zeta_3$; if $3$ does
not divide $n$ then their multiplicities are $[n/3]$ and $[2n/3]$ respectively
\cite{ZarhinCamb}.

 Let $J(C_{f,3})$ be the jacobian of $C_{f,3}$; it is an abelian
variety, whose dimension equals $g(C_{f,3})$. We write $\End(J(C_{f,3}))$ for
the ring of endomorphisms of $J(C_{f,3})$. By Albanese functoriality, $\delta_3$
induces an automorphism of $J(C_{f,3})$ which we still denote by $\delta_3$; it
is known ([11], p. 149, [14], p. 448) that $\delta_3^2+\delta_3+1=0$ in
$\End(J(C_{f,3})$. This gives us an embedding
$$\Z[\zeta_3]\cong \Z[\delta_3]\subset \End(J(C_{f,3})), \ \zeta_3 \mapsto \delta_3$$
(\cite[p. 149]{Poonen}, [14], \cite[p. 448]{Schaefer}).

If $f(x)$ is an odd polynomial of odd degree $n$ then $C_{f,3}$ admits the
involution
$$\delta_2:C_{f,3} \to C_{f,3}, (x,y) \mapsto (-x,-y),$$
which commutes with $\delta_3$. (It has exactly two fixed points if $n$ is not
divisible by $3$.) By Albanese functoriality, $\delta_2$ induces an
automorphism of $J(C_{f,3})$ which we still denote by $\delta_2$ and which
(still) commutes with the automorphism $\delta_3$ of $J(C_{f,3})$. We have
$\delta_2^2=1$ in $\End(J(C_{f,3}))$.

 Let $K$ be a subfield of $\C$ that contains $\sqrt{-3}$ and all
coefficients of $f(x)$, i.e.,
$$f(x)\in K[x]\subset \C[x].$$
Let $\RR_f\subset \C$ be the set of roots of $f(x)$ and $K(\RR_f)$ the
splitting field of $f(x)$ over $K$. Clearly, $K(\RR_f)$ is a finite Galois
extension of $K$. We write $\Gal(f)$ for the (finite) Galois group
$\Gal(K(\RR_f)/K)$. One may view $\Gal(f)$ as a certain permutation subgroup of
the group $\Perm(\RR_f)$ of  permutations of $\RR_f$. If we (choose an order on
$\RR_f$, i.e.,) denote  the roots of $f(x)$ by $\{\alpha_1, \dots , \alpha_n\}$
then we get a group isomorphism between $\Perm(\RR_f)$ and the full symmetric
group $\ST_n$ and $\Gal(f)$ becomes a certain subgroup of $\ST_n$.

It is proven in \cite{ZarhinCrelle,ZarhinMZ} that if $\Gal(f)=\ST_n$ then
$$\End(J(C_{f,3}))=\Z[\delta_3]\cong \Z[\zeta_3].$$
In particular, this allowed us to construct explicitly $g$-dimensional
principally polarized abelian varieties  with endomorphism ring $\Z[\zeta_3]$
for all $g\ge 3$ under an assumption that $3$ does not
divide $g-2$: for example, if  $f(x)=x^{g+1}-x-1$ then $J(C_{f,3})$ is a
$g$-dimensional principally polarized abelian variety with
$\End(J(C_{f,3}))=\Z[\zeta_3]$.  (It is known \cite[p. 42]{SerreGalois} that
the Galois group of the polynomial $x^n-x-1$ over $\Q$ is $\ST_n$ for all
positive integers $n$.)

The aim of this paper is to provide an {\sl explicit} construction of
$m$-dimensional principally polarized abelian varieties, whose endomorphism
ring is $\Z[\zeta_3]$ and $m \ge 5$ is an odd integer that is congruent to $2$
modulo $3$. We construct those abelian varieties (using odd $f(x)$ of degree
$n=2m+1$) as the anti-invariant part of $J(C_{f,3})$ (Prym variety)   with respect
to $\delta_2$, assuming that  $\Gal(f)$ coincides with the Weyl group
$\W(\D_m)$ of the root system
 $\D_m$ in the following sense. Since $f(x)$ is  odd and without multiple roots, there exist
$m$ distinct non-zero roots $\{\beta_1, \dots , \beta_m\}$ of $f(x)$ such that
($\beta_i \ne \pm \beta_j$ if $i\ne j$ and)
$\RR_f$  coincides with the set $\{0\} \cup \{\pm \beta_1, \dots , \pm \beta_m
\}\subset \bar{K}$. Then $\W(\D_m)$ may be defined as
 the group of  permutations of $\RR_f$ of the form
 $$0\mapsto 0, \ \beta_i \mapsto \epsilon_i \beta_{s(i)}, \ -\beta_i \mapsto -\epsilon_i \beta_{s(i)}$$
 where $s \in \ST_m$ is an arbitrary permutation on $m$ letters and signs
$\epsilon_i=\pm 1$ satisfy the condition $\prod_{i=1}^m \epsilon_i =1$. Let us
consider the $m$-dimensional $\F_3$-vector space of {\sl odd} functions
$$V_f^{-}:=\{\phi:\RR_f \to \F_3\mid \phi(-\alpha)=-\phi(\alpha) \ \forall
\alpha\in \RR_f\}$$ provided with the natural structure of Galois module.

Our main result is the following statement.

\begin{thm}
\label{main0} Suppose that $k$ is a nonnegative integer and $m=6k+5\ge 5$. Suppose that $n=2m+1=12k+11$
and $f(x)$ is an odd polynomial of degree $n$ without multiple roots. Then:

\begin{itemize}
\item[(i)]
\begin{enumerate}
\item
\begin{itemize}
\item[(A)]
$P(C_{f,3}):=(1-\delta_2)J(C_{f,3})$ is an $m$-dimensional $\delta_3$-invariant
abelian subvariety in $J(C_{f,3})$. In particular, the embedding
$\Z[\delta_3]\subset \End(J(C_{f,3}))$ induces the embedding
$$\Z[\zeta_3]\cong \Z[\delta_3]\hookrightarrow \End(P(C_{f,3})).$$
%%%%%%%%%%%%%%%%%%%%%%%%%%%%%%%%%%
\item[(B)]
If one restrict the canonical principal polarization on $J(C_{f,3})$ to
$P(C_{f,3})$ then the induced polarization is twice a principal polarization on
$P(C_{f,3})$ and this principal polarization is $\delta_3$-invariant.
\item[(C)]
The principally
polarized abelian variety $P(C_{f,3})$ is not isomorphic to the canonically
polarized jacobian of a smooth projective curve.
%%%%%%%%%%%%%%%%%%%%%%%%%%%%%%%%%%%%%%%%%
\end{itemize}

\item By functoriality, $\delta_3$ induces the linear operator
$$\delta_{3,P}^{*}: \Omega^1(P(C_{f,3})) \to \Omega^1(P(C_{f,3}))$$
in the space of differentials of the first kind on $P(C_{f,3})$. Its spectrum
consists of eigenvalues $\zeta_3^{-1}$ of multiplicity $2k+1$ and $\zeta_3$  of
multiplicity $4k+4$.
\end{enumerate}
 \item[(ii)] Suppose that $K$ is a subfield
of $\C$ that contains $\sqrt{-3}$ and all coefficients of $f(x)$. Then:
\begin{itemize}
\item[(a)]
 The abelian variety $P(C_{f,3})$ and its automorphism $\delta_3$ are defined over
$K$. In addition, the Galois submodule $P(C_{f,3})^{\delta_3}$ of
$\delta_3$-invariants of $P(C_{f,3})(\bar{K})$ is canonically isomorphic to
 $V_f^{-}$.
  \item[(b)]
 Assume
additionally that $\Gal(f)$ coincides with $\W(\D_m)$. Then:
\begin{itemize}
\item[(b1)] $\End(P(C_{f,3}))=\Z[\zeta_3]$. In particular, $P(C_{f,3})$ is an
absolutely simple abelian variety.
 \item[(b2)] The abelian variety
$P(C_{f,3})$ is isomorphic neither to the jacobian of a smooth projective curve
nor to a product of jacobians of smooth projective curves (even if one ignore
polarizations).
\end{itemize}
\end{itemize}
\end{itemize}
\end{thm}

\begin{ex}
\label{elkin} Let $m=5$ and $f(x):= x(x^{10} - x^2 -1)$. One may check (see
Example \ref{magma} below) that the Galois group of $f(x)$ over
$K=\Q(\sqrt{-3})$ is $\W(\D_5)$ . It follows that
$$\End(P(C_{f,3}))=\Z[\zeta_3].$$
\end{ex}

\begin{rem}
If $m=5$ then the $5$-dimensional Prym varieties $P(C_{f,3})$ appear as
intermediate jacobians of certain cubic threefolds \cite{MT}. (See also
\cite{ZarhinMF}.)
\end{rem}

\begin{rem}
A complete list of those (generalized) Prym varieties that are isomorphic, as
principally polarized abelian varieties, to jacobians of smooth projective
curves or to products of them was given by V.V. Shokurov
\cite{SlavaInv,SlavaIzv}. In the course of the proof of Theorem
\ref{main0}(i)(1)(C)) we use a different approach based on the study of the action of
the period $3$ automorphism on the differentials of the first kind
\cite{ZarhinMF}.
\end{rem}

The paper is organized as follows. In Section \ref{group} we discuss permutation modules
 related to Galois groups of odd polynomials. In Section \ref{geom} we study trigonal
 jacobians and prymians and prove the main result.

I am grateful to V.V. Shokurov for useful discussions. My special thanks go to Dr.
Arsen Elkin,  who did computations  with MAGMA related to Example \ref{magma}.

\section{Galois groups of odd polynomials and permutation modules}
\label{group}

Let $K$ be a field of characteristic zero, $\bar{K}$  its algebraic closure and
$\Gal(K)=\Aut(\bar{K}/K)$ its absolute Galois group. Let $\gamma\in K$ be a
primitive cube root of unity.

\begin{sect} {\bf Galois groups of odd polynomials}.
Let $n=12k+11$ be a positive integer that is congruent to $11$ modulo $12$,
$f(x)\in K[x]$ a degree $n$ {\sl odd} polynomial {\sl without multiple roots}
and with {\sl non-zero} constant term. Let us put $m=6k+5$. Then there exist
$m$ distinct non-zero roots $\{\beta_1, \dots , \beta_m\}$ of $f(x)$ such that
the $n$-element set $\RR_f$ of roots of $f(x)$ coincides with $\{0\} \cup \{\pm
\beta_1, \dots , \pm \beta_m \}\subset \bar{K}$ of all roots of $f(x)$.
Clearly, $\RR_f$ is Galois-stable. We write $\Perm(\RR_f)$ for the group of
permutations of the $n$-element set $\RR_f$. Let $\Gal(f)$ be the image of
$\Gal(K)$ in $\Perm(\RR_f)$. If $K(\RR_f)$ is the splitting field of $f(x)$ obtained by
adjoining to $K$ all elements of $\RR_f$ then $K(\RR_f)/K$ is a finite Galois
extension and  $\Gal(f)$ is canonically isomorphic to the Galois group
$\Gal(K(\RR_f)/K)$. Let $\Perm_0(\RR_f)$ be the subgroup of $\Perm(\RR_f)$ that
consists of all permutations of the form
 $$0\mapsto 0, \ \beta_i \mapsto \epsilon_i \beta_{s(i)}, \ -\beta_i \mapsto -\epsilon_i \beta_{s(i)}$$
 where $s \in \ST_m$ is an arbitrary permutation on $m$ letters and
$\epsilon_i=\pm 1$. Clearly,
$$\Gal(f)\subset \Perm_0(\RR_f) \subset \Perm(\RR_f).$$
We write $\W(\D_m)$ for the index $2$ subgroup of $\Perm_0(\RR_f)$, whose elements are characterized by
the condition
$\prod_{i=1}^m \epsilon_i =1$. We have
$$\W(\D_m)\subset \Perm_0(\RR_f) \subset \Perm(\RR_f).$$
\end{sect}

Since $0$ is a simple root of $f(x)$ we have $f(x)=x\cdot  h (x)$ where $h(x)$
is an {\sl even} polynomial of even degree $2m$, whose set of roots $\RR_{h}$
is $\{\pm \beta_1, \dots , \pm \beta_m\}$; in particular, $h(0)\ne 0$.

\begin{rem}
 Clearly,  $(\prod_{i=1}^m \beta_i)^2 = -h(0)$ (recall that
 $m$ is odd).
It follows  that $\Gal(f)=\Gal(h)\subset \W(\D_m)$ if and only if $-h(0)$
is a square in $K$.
\end{rem}

\begin{ex}
\label{magma} Let $$m=5, h(x)=x^{10}-x^2-1 \in \Q[x], f(x)=x \cdot h(x)= x
(x^{10}-x^2-1).$$ Since $1=-h(0)$ is a square in $\Q$, the Galois group
$\Gal(h/\Q)$ of $h(x)$ over $\Q$ is a subgroup of $\W(\D_5)$. Using Magma
\cite{MAGMA}, one obtains that the order of $\Gal(h/\Q)$ is $2^4 \cdot 5!$.
Since the order of $\W(\D_5)$ is also $2^4 \cdot 5!$, we conclude that
$\Gal(h/\Q)=\W(\D_5)$. One may check that the derived (sub)group
$G_1:=(\W(\D_5),\W(\D_5))$ is a {\sl perfect} (normal) subgroup of index $2$ in
$\W(\D_5)$. It follows that the splitting field $\Q(\RR_h)$ of $h(x)$ over $\Q$
contains exactly one quadratic subfield. In order to determine this subfield,
notice that if $\{\pm \beta_1, \dots , \pm \beta_5\}$ is the set of roots of
$h(x)=x^{10}-x^2-1$ then $\{\beta_1^2, \dots \beta_5^2\}$ is the set of roots
of $x^5-x-1$. Using Magma \cite{MAGMA}, one obtains that the discriminant of
$x^5-x-1$ is $19 \times 151$ and therefore $\Q(\RR_h)$ contains $\Q(\sqrt{19
\times 151})$. It follows that $\Q(\RR_h)$ does {\sl not} contain
$K=\Q(\sqrt{-3})$ and therefore $\Q(\RR_h)$ and $K$ are linearly disjoint over
$\Q$. This implies that the Galois group of $h(x)$ over $K$ also coincides with
$\W(\D_5)$ and therefore the Galois group of $f(x)$ over $K$ also  coincides
with $\W(\D_5)$.
\end{ex}

\begin{defn}
Let $\Perm(\RR_h)$ be the group of permutations of the $2m$-element set
$\RR_h$.
 Let $\GG$ be a permutation subgroup in $\ST_m$. We write $2^m \cdot \GG\subset \Perm(\RR_h)$ for
 the subgroup of all permutations of
the form
$$(s; \epsilon_1, \dots , \epsilon_m):
\beta_i \mapsto \epsilon_i \beta_{s(i)}, \ -\beta_i \mapsto -\epsilon_i
\beta_{s(i)}$$
 where $$s \in \GG, \ \epsilon_i=\pm 1 .$$ We write $2^{m-1} \cdot \GG$ for
 the index two subgroup in $2^m \cdot \GG$, whose elements are characterized by the
 condition $\prod_{i=1}^m \epsilon_i =1$.
\end{defn}

\begin{ex}
\item[(i)] The group $2^m \cdot \{1\}$ coincides with the group of all
permutations of the form
$$\beta_i \mapsto \epsilon_i \beta_{i}, \ -\beta_i \mapsto -\epsilon_i
\beta_{i}$$  where $ \epsilon_i=\pm 1$ while $2^{m-1} \cdot \{1\}$ corresponds
to its index $2$ subgroup, whose elements are characterized by the condition
 $\prod_{i=1}^m \epsilon_i =1$. The groups $2^m \cdot \{1\}$ and  $2^{m-1} \cdot
 \{1\}$ are exponent $2$ commutative groups of order $2^m$ and $2^{m-1}$
 respectively.
 \item[(ii)]
 Let us identify $\Perm(\RR_h)$ with the stabilizer of $0$ in $\Perm(\RR_f)$. Then
$2^m\cdot \ST_m$ coincides with $\Perm_0(\RR_f)$ and $2^{m-1}\cdot \ST_m$
coincides with $\W(\D_m)$.
\end{ex}

\begin{rem}
Clearly, the natural map $(s; \epsilon_1, \dots , \epsilon_m)\mapsto s$ give
rise to the surjective group homomorphisms
$$\kappa^{0}_h: 2^m \cdot \GG \twoheadrightarrow \GG, \ \kappa_h: 2^{m-1} \cdot \GG \twoheadrightarrow
\GG,$$ whose kernels are $2^m \cdot \{1\}$ and $2^{m-1} \cdot \{1\}$
 respectively.
\end{rem}

\begin{rems}
\label{orbits} Suppose that there exists a  permutation group $\GG \subset
\ST_m$ such that $\Gal(h) =2^{m-1} \cdot \GG$.
 Then:

\begin{itemize}
\item[(i)] The kernel of the surjective group homomorphism $\kappa_h:
\Gal(h)=2^{m-1} \cdot \GG  \twoheadrightarrow \GG$  is the commutative (normal sub)group
 $2^{m-1} \cdot \{1\}$ of exponent $2$.
 \item[(ii)]
 Let $G_1$ be a {\sl normal} subgroup in $\Gal(h)$ of {\sl odd} index say, $r$. Then
 $G_1$ contains $2^{m-1} \cdot \{1\}$ and the surjectivity of $\kappa_h$ implies that
 $$\kappa_h(G_1) \cong G_1/(2^{m-1} \cdot \{1\})$$ is a normal subgroup of index $r$ in $\GG$. This
 implies that if $\GG$ does not contain a normal subgroup of odd index (except
 $\GG$ itself) then $\Gal(h)$ also does not contain a normal subgroup of odd index (except
 $\Gal(h)$ itself).
\item[(iii)]
\begin{enumerate}
\item If $\GG$ is a transitive subgroup of $\ST_m$ then $2^{m-1} \cdot \GG$ is
a transitive subgroup of $\Perm(\RR_h)$. This means that $\Gal(h)$ is a
transitive subgroup of $\Perm(\RR_h)$, i.e., $h(x)$ is irreducible over $K$.

 \item
 Suppose that $\GG$ is a doubly transitive subgroup of $\ST_m$ and let
$\GG_1$ is the stabilizer of $1$ in $\GG$. Then $\GG_1$ has exactly two orbits
in $\{1, \dots , m\}$: namely, $\{1\}$ and the rest.
 Let $\Gal(h)_1$ be the stabilizer of
$\beta_1$ in $\Gal(h)$. Then one may easily check that $\Gal(h)_1$ has exactly $3$ orbits in
$\RR_h$: namely, $\{\beta_1\}$, $\{-\beta_1\}$ and the rest.
\end{enumerate}
\end{itemize}
\end{rems}

\begin{sect}
{\bf Permutation modules}. Let $V_{{f}}$ be the
 $2m$-dimensional $\F_3$-vector space of functions $$\phi:{\RR}_{f}\to
 \F_3, \ \sum_{\alpha\in {\RR}_{{f}}}\phi(\alpha)=0.$$
 The space $V_{{f}}$ carries the natural structure of Galois module
 induced by the Galois action on ${\RR}_{{f}}$.

 Let $\F_3^{\RR_h}$ be the
$2m$-dimensional $\F_3$-vector space of all functions $\phi:{\RR}_{h}\to
 \F_3$.
  It carries the natural structure of Galois module. We write $1_{\RR_h}$ for the
  (Galois-invariant) constant function $1$.

  The map that assigns
 to a $\F_3$-valued function on $\RR_f$ its restriction to $\RR_h$ gives rise to
 the isomorphism $V_f \to \F_3^{\RR_h}$ of Galois modules. (One may extend a
 function $\phi$ on $\RR_h$ to $\RR_f=\{0\}\cup \RR_h$ by putting
 $$\phi(0):=-\sum_{\alpha\in \RR_h}\phi(\alpha).)$$
 The Galois module $V_{{f}}$ splits into a
direct sum of the Galois submodules of odd and even functions
$$V_{{f}}=V_{{f}}^{-}\oplus V_{{f}}^{+}$$
where
$$V_{{f}}^{+}=\{\phi:{\RR}_{f}\to \F_3, \sum_{\alpha \in \RR_f} \phi(\alpha)=0, \ \phi(\alpha)= \phi(-\alpha) \ \forall \alpha\},$$
$$V_{{f}}^{-}=\{\phi:{\RR}_{f}\to \F_3, \ \phi(\alpha)= -\phi(\alpha) \ \forall \alpha\}.$$
(The sum of values of an odd function is always zero.) Clearly,
 $\phi(0)=0$ for all $\phi \in V_{{f}}^{-}$. It follows that
$$\dim_{\F_3}(V_{{f}}^{-})=m.$$
\end{sect}

\begin{lem}
\label{centralizer} Suppose that  there exists a doubly transitive permutation
group $\GG \subset \ST_m$ such that $\Gal(h)=2^{m-1}\cdot\GG$.
 Then $\End_{\Gal(K)}(V_{{f}}^{-})=\F_3$.
\end{lem}

\begin{proof}
By Remark \ref{orbits}(iii)(1), $\Gal(h)$ acts transitively on $\RR_h$.

 Let $W_h^{+}$ and $W_h^{-}$ be the subspaces of even and odd
functions respectively in $\F_3^{\RR_h}$. Clearly, they both are Galois
submodules in $\F_3^{\RR_h}$ and $$W_h^{-}\oplus W_h^{+} =\F_3^{\RR_h}.$$ It
is also clear that the Galois modules $W_h^{-}$ and $V_{{f}}^{-}$ are
isomorphic. So, it suffices to check that
$$\End_{\Gal(K)}(W_{{h}}^{-})=\F_3.$$
In order to do that, notice that $\#(\RR_h)=2m=n-1=12k+10$ is {\sl not}
divisible by $3$. This implies that the submodule $\F_3 \cdot 1_{\RR_h}$ of constant functions is a
direct summand of $W_h^{+}$ and $\F_3^{\RR_h}$ splits into a direct sum of
Galois modules
$$\F_3^{\RR_h}=W_h^{-}\oplus W_h^{+}=W_h^{-}\oplus \F_3 \cdot 1_{\RR_h}
\oplus W_h^{+,0}$$ where $W_h^{+,0}$ is the Galois (sub)module of even
functions, whose sum of values is zero. Clearly,
$$\dim_{\F_3}\End_{\Gal(K)} (\F_3^{\RR_h})\ge $$
$$\dim_{\F_3}\End_{\Gal(K)}(W_h^{-}) + \dim_{\F_3}\End_{\Gal(K)}(\F_3 \cdot
1_{\RR_h}) + \dim_{\F_3}\End_{\Gal(K)}(W_h^{+,0}) \ge $$
$$\dim_{\F_3}\End_{\Gal(K)}(W_h^{-}) +1 +1.$$ So, if we prove that
$\dim_{\F_3}\End_{\Gal(K)} (\F_3^{\RR_h})=3$ then we are done. Since the image
of $\Gal(K)$ in $\Aut_{\F_3}(\F_3^{\RR_h})$ coincides with
$$\Gal(h)\subset \Perm(\RR_h)\subset \Aut_{\F_3}(\F_3^{\RR_h}),$$
we have
$$\End_{\Gal(K)} (\F_3^{\RR_h})=\End_{\Gal(h)} (\F_3^{\RR_h}).$$
So, in order to prove the Lemma, it suffices to check that
$$\dim_{\F_3}(\End_{\Gal(h)} (\F_3^{\RR_h}))=3.$$
By Lemma 7.1 of \cite{Passman}, $\dim_{\F_3}(\End_{\Gal(h)} (\F_3^{\RR_h}))$
coincides with the number of orbits in $\RR_h$ of the stabilizer  in $\Gal(h)$
of any root of $h(x)$. But the number of orbits is $3$ (see Remark
\ref{orbits}(iii)(2)). This ends the proof.
\end{proof}

\section{Cyclic covers, jacobians and prymians}
\label{geom}

If $X$ is an abelian variety over $\bar{K}$ then we write $\End(X)$ for the
ring of its $\bar{K}$-endomorphisms and $\End^0(X)$ for the
corresponding $\Q$-algebra $\End(X)\otimes\Q$. If $X$ is defined over $K$ then
we write $\End_K(X)$ for the ring of its $K$-endomorphisms.

As above $f(x)=x\cdot h(x)\in K[x]$ is an odd polynomial of degree $n=2m+1=12k+11$
without multiple roots. We keep all the notation of the previous Section.

\begin{sect}
{\bf Trigonal curves}. Hereafter we assume that $K$  contains $\sqrt{-3}$. Let
us consider (the smooth projective model of) the trigonal curve
$$C_{f,3}: y^3=f(x)$$
of genus $n-1=12k+10$. The curve $C_{f,3}$ admits commuting periodic
automorphisms
$$\delta_2:(x,y)\mapsto (-x,-y)$$
and
$$\delta_3:(x,y)\mapsto (x,\gamma y)$$
of period $2$ and $3$ respectively.

The regular map of curves
$$\pi:C_{f,3}\to \P^1, \ (x,y)\mapsto x$$
has degree $3$ and ramifies exactly at $0$, the $2m$-element set set
$\{\alpha\mid \alpha\in \RR_f\}$ and $\infty$. (Notice that $3$ does {\sl not}
divide $2m+1$.) Clearly, all  branch points of $\pi$ in $C_{f,3}$ are
$\delta_3$-invariant. By abuse of notation, we denote $\pi^{-1}(\infty)$ by
$\infty$. Let us put
$$B=\pi^{-1}(\RR_f)=\{(\alpha,0)\mid \alpha\in \RR_f\}\subset C(\bar{K}).$$
Clearly, all elements of $B$ are $\delta_3$-invariant. On the other hand, if
$P= (\alpha,0)\in B$ then $\delta_2(P)=(-\alpha,0)\in B$.

The automorphism $\delta_2:C_{f,3} \to C_{f,3}$ has exactly two fixed points,
namely, $\pi^{-1}(0)$ and $\pi^{-1}(\infty)$. One may easily check that the
quotient $\tilde{C}_{f,3}=C_{f,3}/(1,\delta_2)$ is a smooth (irreducible)
projective curve (compare with Lemma 1.2, its proof and Corollary 1.3 in
\cite{ZarhinMF}) and $C_{f,3}\to \tilde{C}_{f,3}$ is a double covering that is
ramified at exactly two points, namely the images of $\pi^{-1}(0)$ and
$\pi^{-1}(\infty)$. The Hurwitz formula implies that the genus of
$\tilde{C}_{f,3}$ is $m$.

Since
$$[n/3]=4k+3, \  [2n/3]=8k+7,$$
it follows from (\cite{ZarhinCrelle}, \cite[Remarks 3.5 and 3.7]{ZarhinCamb}) that the $(n-1)$-dimensional
$\bar{K}$-vector space $\Omega^1(C_{f,3})$ of differentials of the first kind
on $C_{f,3}$ has a basis
$$\left\{ x^i \frac{dx}{y}, \ 0\le i \le 4k+2; \ x^j \frac{dx}{y^2}, \ 0\le j \le
8k+6 \right\}.$$ If
$$\delta_2^{*}:\Omega^1(C_{f,3})\to \Omega^1(C_{f,3}), \ \delta_3^{*}:\Omega^1(C_{f,3})\to
\Omega^1(C_{f,3})$$ are the automorphisms induced by $\delta_2$ and $\delta_3$
respectively then
$$\delta_3^{*}(x^i \frac{dx}{y})=\gamma^{-1}x^i \frac{dx}{y}, \ \delta_3^{*}(x^j
\frac{dx}{y^2})=\gamma^{-2}x^j \frac{dx}{y^2}=\gamma x^j \frac{dx}{y^2},$$
$$\delta_2^{*}(x^i \frac{dx}{y})=(-1)^i x^i \frac{dx}{y}, \ \delta_2^{*}(x^j
\frac{dx}{y^2})=(-1)^{j+1}x^j \frac{dx}{y^2};$$ in particular, the
 basis consists of eigenvectors with respect to $\delta_2^{*}$ and $\delta_3^{*}$.
It follows that the subspace $\Omega^1(C_{f,3})^{-}$ of
$\delta_2$-anti-invariants is $m$-dimensional and admits a basis
$$\left\{ x^{2i+1} \frac{dx}{y}, \ 0\le i \le 2k; \ x^{2j} \frac{dx}{y^2} \ 0\le j \le
4k+3 \right\}.$$
\end{sect}

\begin{sect}
{\bf Trigonal jacobians}. Let $J(C_{f,3})$ be the jacobian of $C_{f,3}$: it is
a $(n-1)$-dimensional abelian variety that is defined over $K$. By Albanese
functoriality, $\delta_2$ and $\delta_3$ induce the $K$-automorphisms of
 $J(C_{f,3})$ that we still denote by $\delta_2$ and $\delta_3$ respectively.
 We have
 $$\delta_2^2=1, \ \delta_3^2+\delta_3+1=0$$
 where all the equalities hold in $\End(J(C_{f,3}))$.
 The latter equality gives rise to the embedding
 $$\Z[\zeta_3] \hookrightarrow \End_K(J(C_{f,3})), \ \zeta_3 \mapsto
 \delta_3$$
 of the cyclotomic ring $\Z[\zeta_3]$ into the  ring of $K$-endomorphisms of $J(C_{f,3})$.

 Let $j:C_{f,3} \hookrightarrow J(C_{f,3})$ be canonical embedding of $C_{f,3}$ into
 its jacobian normalized by the condition $j(\infty)=0$, i.e., $j$ sends a point
 $P \in C_{f,3}(\bar{K})$ to the linear equivalence class of the divisor
 $(P)-(\infty)$. Clearly, $j$ is $\delta_3$-equivariant and
 $\delta_2$-equivariant.

 Let me remind the description of the Galois (sub)module $J(C_{f,3})^{\delta_3}$ of
 $\delta_3$-invariants in $J(C_{f,3})(\bar{K})$.  The Galois modules
$V_{{f}}$ and $J(C_{f,3})^{\delta_3}$ are canonically isomorphic
\cite{Schaefer} (see also \cite{ZarhinMiami}). Namely, let
$$\Z_B^0=\{\sum_{P\in B}a_P(P)\mid a_P\in \Z, \ \sum_{P\in B}a_P=0 \}$$
be the group of degree zero divisors on $C_{f,3}$ with support in $B$. The free
commutative group $\Z_B^0$ carries the natural structure of Galois module.
Clearly, the Galois module $\Z_B^0/3 \Z_B^0$ is canonically isomorphic to
$V_{{f}}$: a divisor $\sum_{P\in B}a_P(P)$ gives rise tho the function $\alpha
\mapsto a_P \mod 3$ where $P=(\alpha,0)$. Since $B$ is $\delta_2$-invariant,
the map
$$D_2: \sum_{P\in B}a_P(P)\mapsto \sum_{P\in B}a_P(\delta_2 P)$$
is an automorphism of of the Galois module $\Z_B^0$ that induces the
automorphism of $\Z_B^0/3 \Z_B^0=V_{{f}}$ that sends a function $\alpha \to
\phi(\alpha)$ to the function $\alpha \to \phi(-\alpha)$. (We still denote this
automorphism of $V_{{f}}$ by $D_2$.)  Notice that
$$V_{{f}}^{-}=(1-D_2)V_{{f}}, \ V_{{f}}^{+}=(1+D_2)V_{{f}}$$
(recall that $V_f$ is the $\F_3$-vector space.) In other words, $V_{{f}}^{+}$
and $V_{{f}}^{-}$ are eigenspaces of $D_2$ that correspond to eigenvalues $1$
and $-1$ respectively.

Let us consider the natural map
$$\cl: \Z_B^0 \to J(C_{f,3})(\bar{K})$$
that sends a divisor $\sum_{P\in B}a_P(P)$ to (its linear equivalence class,
i.e., to) $$\sum_{P\in B}a_P j(P) \in J(C_{f,3})(\bar{K}).$$ It turns out that
$\cl(\Z_B^0)=J(C_{f,3})^{\delta_3}$ and the kernel of $\cl$ coincides with
$3\cdot \Z_B^0$. This gives rise to the natural isomorphism of Galois module
$\Z_B^0/3 \Z_B^0$ and $J(C_{f,3})^{\delta_3}$ and we get the natural
isomorphisms of Galois modules
$$\overline{\cl}: V_{{f}}=\Z_B^0/3 \Z_B^0\cong J(C_{f,3})^{\delta_3}.$$
Since $\delta_2$ commutes with $\delta_3$, the Galois submodule
$J(C_{f,3})^{\delta_3}$ is $\delta_2$-invariant. It follows from the explicit
description of $\cl$ and $D_2$ that if $\bar{\cl}(\phi)=P \in
J(C_{f,3})^{\delta_3}$ then $\delta_2 P$ is the image (under $\overline{\cl}$)
of the function $\alpha \to \phi(-\alpha)$. In other words,
$$\overline{\cl}(D_2 \phi)=\delta_2 \overline{\cl}(\phi) \ \forall \phi \in
V_f.$$ It follows that the restriction of $\overline{\cl}$ to $V_{{f}}^{-}$
gives us the isomorphism of Galois modules
$$\overline{\cl}: V_{{f}}^{-} \cong \{P \in J(C_{f,3})^{\delta_3}\mid
\delta_2 P=-P\}.$$ This implies that
$$\{P \in J(C_{f,3})^{\delta_3}\mid
\delta_2 P=-P\}=\overline{\cl}(V_{f}^{-})=\overline{\cl}((1-D_2)V_{f})
 =(1-\delta_2)J(C_{f,3})^{\delta_3}.$$
\end{sect}

\begin{sect}
\label{Triprym} {\bf Trigonal prymians}.
 Let us consider the Prym variety
 $$P(C_{f,3})=(1-\delta_2)J(C_{f,3})\subset J(C_{f,3}).$$
 If one restrict the canonical principal polarization on $J(C_{f,3})$ to
$P(C_{f,3})$ then the induced polarization is twice a principal polarization on
$P(C_{f,3})$ \cite[Sect. 3, Cor. 2]{MumfordP}.
%%%%%%%%%%%%%%%%%%%%%%%%%%%%%%%%%%
Obviously, the principal polarization on
$P(C_{f,3})$ is $\delta_3$-invariant.
%%%%%%%%%%%%%%%%%%%%%%%%%%%%%%%%%%%%%%%%%
It is also clear \cite[Sect. 3,
Cor. 2]{MumfordP} that $P(C_{f,3})$ coincides with the identity component of
the surjective map of jacobians $J(C_{f,3})\to J(\tilde{C}_{f,3})$; in
particular, it is a $m$-dimensional abelian variety that is defined over $K$.
Clearly, the abelian subvariety $P(C_{f,3})$ is $\delta_3$-invariant. Therefore
we may and will consider $\delta_3$ as the $K$-automorphism of $P(C_{f,3})$. Still
$\delta_3^2+\delta_3+1=0$ in $\End(P(C_{f,3}))$. As above, this induces an
embedding $$\Z[\zeta_3] \hookrightarrow \End(P(C_{f,3})), \ \zeta_3 \mapsto
 \delta_3.$$
On the other hand, $1+\delta_2$ kills $P(C_{f,3})$, because
$$0=1-\delta_2^2=(1+\delta_2)(1-\delta_2)\in \End(J(C_{f,3}))$$ and
$P(C_{f,3})(\bar{K})=(1-\delta_2)(J(C_{f,3}))$. This implies that
$$\delta_2 P = -P \ \forall P \in P(C_{f,3})(\bar{K}).$$
Let us consider the Galois (sub)module $P(C_{f,3})^{\delta_3}$ of
$\delta_3$-invariants in  $P(C_{f,3})(\bar{K})$. Clearly,
$$P(C_{f,3})^{\delta_3}\subset \{P \in J(C_{f,3})^{\delta_3}\mid
\delta_2 P=-P\}.$$ Since the latter group coincides with
$(1-\delta_2)J(C_{f,3})^{\delta_3}$, we conclude that
$$P(C_{f,3})^{\delta_3}= \{P \in J(C_{f,3})^{\delta_3}\mid
\delta_2 P=-P\}.$$
It follows that the Galois modules $P(C_{f,3})^{\delta_3}$
and $V_f^{-}$ are canonically isomorphic.

Let us put $$\Oc=\Z[\zeta_3], \ \lambda=(1-\zeta_3)\Oc, \
E=\Oc\otimes\Q=\Q(\zeta_3)=\Q(\sqrt{-3}).$$ Then the residue field
$$\Oc/\lambda=\F_3.$$  Recall that we have the natural homomorphism
$$\Oc=\Z[\zeta_3]\hookrightarrow \End_K(P(C_{f,3})), \zeta_3\mapsto \delta_3.$$
This implies that
$$P(C_{f,3})^{\delta_3}=P(C_{f,3})_{\lambda}$$
and therefore the Galois modules  $P(C_{f,3})_{\lambda}$ and $V_{f}^{-}$ are
canonically isomorphic. In particular,
$$\dim_{\F_3}(P(C_{f,3})_{\lambda})=m.$$
On the other hand, it is well known \cite{Shimura,Ribet2,ZarhinMZ} that $P(C_{f,3})_{\lambda}$ is a free
$\Oc/\lambda$-module of rank $2\dim(P(C_{f,3}))/[E:\Q]$. Since
$\Oc/\lambda=\F_3$ and $[E:\Q]=2$, we get another proof of the equality $\dim(P(C_{f,3}))=m$.
Notice that
 $$\dim(P(C_{f,3}))=m=\dim_{\F_3}(\Omega^1(C_{f,3})^{-}).$$
\end{sect}

\begin{rem}
\label{multprime}
 Taking into account that $\dim(P(C_{f,3}))=m$ and applying Theorem 3.10 of \cite{ZarhinMiami} to
$$ Y=J(C_{f,3}), Z=P(C_{f,3}), \delta=\delta_2, P(t)=1-t,$$
we obtain that $1-\delta_2:J(C_{f,3})\twoheadrightarrow P(C_{f,3})\subset
J(C_{f,3})$ induces (on differentials of the first kind) an isomorphism
$$(1-\delta_2)^{*}:\Omega^1(P(C_{f,3}))\cong \Omega^1(C_{f,3})^{-}\subset
\Omega^1(J(C_{f,3}))$$
 and this isomorphism is $\delta_3$-equivariant. It
follows easily that $\delta_3$ induces a linear operator in
$\Omega^1(P(C_{f,3}))$, whose spectrum consists of eigenvalues $\gamma^{-1}$ of
multiplicity $2k+1$ and $\gamma=\gamma^{-2}$ of multiplicity $4k+4$. Clearly,
the numbers $2k+1$ and $4k+4$ are {\sl relatively prime}.
\end{rem}

\begin{thm}
\label{main}
 Assume that there exists a doubly transitive permutation group $\GG \subset
 \ST_m$ that enjoys the following properties:

 \begin{itemize}
 \item[(i)]
 $\GG$ does not contain a normal subgroup, whose index divides $m$ (except $\GG$
 itself).
\item[(ii)] $\Gal(h)$ contains $2^{m-1}\cdot \GG$.
 \end{itemize}
 Then $\End(P(C_{f,3}))=\Z[\zeta_3]$.
 In particular, $P(C_{f,3})$ is an absolutely simple abelian variety.
\end{thm}

\begin{proof}
Enlarging $K$ if necessary, we may and will assume that $\Gal(h)= 2^{m-1}\cdot
\GG$. Identifying $\Perm(\RR_h)$ with the stabilizer of $0$ in $\Perm(\RR_f)$,
we obtain that
$$\Gal(f)=\Gal(h)=2^{m-1}\cdot\GG.$$
Since the Galois modules  $P(C_{f,3})_{\lambda}$ and $V_f^{-}$ are isomorphic,
it follows from Lemma \ref{centralizer} that
$\End_{\Gal(K)}(P(C_{f,3})_{\lambda})=\F_3$.
 Now Theorem \ref{main} follows from  Remark \ref{multprime} and Theorem
3.12(ii)(2) of \cite{ZarhinMZ} applied to $X=P(C_{f,3}), E=\Q(\sqrt{-3}),
\Oc=\Z[\zeta_3], \lambda=(1-\zeta_3)\Oc$.
\end{proof}

\begin{ex}
Let $m=6k+5$ be a positive integer that is congruent to $5$ modulo $6$.
%Let $\GG\subset \ST_m$ be a doubly transitive subgroup that does not contain a normal subgroup,
%whose index divides $m$ except $\GG$ itself. (E.g., $\GG=\ST_m$ or the alternating group $\A_m$.)
  Let $L$ be the field of rational functions $\C(t_1, \dots , t_m)$ in $m$ independent
 variables $t_1, \dots , t_m$ over $\C$.
 One may realized $2^{m}\cdot\ST_m$ as the following  group of (linear) automorphisms of $L$:
 $$(s; \epsilon_1, \dots , \epsilon_m):
t_i \mapsto \epsilon_i t_{s(i)}, \ i=1, \dots m$$
 where $$s \in \ST_m, \ \epsilon_i=\pm 1 .$$
 Let $K$ be the subfield of $2^{m}\cdot\ST_m$-invariants in $L$. Clearly, $L/K$ is a finite
 Galois extension with Galois group $2^{m}\cdot\ST_m$. In particular, $\bar{L}=\bar{K}$.
Since $m \ge 5$, the only normal subgroups in $\ST_m$ are the subgroup $\{1\}$ of even index
$m!$, the alternating (sub)group $\A_m$ of index $2$ and $\ST_m$ itself.

 The  even degree $2m$ polynomial
 $$h(x) =\prod_{i=1}^m (x^2-t_i^2)=\prod_{i=1}^m (x-t_i) \prod_{i=1}^m (x+t_i)$$
 lies in $K[x]$ and its splitting field coincides
 with $L$. It follows that $\Gal(h)=2^{m}\cdot\ST_m$.  Applying Theorem \ref{main} to the odd
 degree $(2m+1)$ polynomial
$$f(x):= x \cdot h(x)=x \cdot \prod_{i=1}^m (x^2-t_i^2),$$
 we conclude that the
 endomorphism ring (over $\bar{L}$) of the $m$-dimensional prymian $P(C_{f,3})$ coincides
  with $\Z[\zeta_3]$.
\end{ex}

\begin{proof}[Proof of Theorem \ref{main0}]
The assertions (i) (except (i)(1)(C) ) and (ii)(a) are already proven in Subsection \ref{Triprym} and
Remark \ref{multprime}.  Since $\ST_m$ is the doubly transitive permutation
group that does not contain normal subgroups of odd index (except $\ST_m$
itself) and $\W(\D_m)=2^{m-1}\cdot \ST_m$, the assertion (ii)(b1) follows from
Theorem \ref{main} applied to $\GG=\ST_m$.

In order to prove the assertion (i)(1)(C), notice that
$$3\cdot \mid (2k+1)-(4k+4)\mid =6k+9>(6K+5)+2=m+2=\dim(P(C_{f,3}))+2.$$
Now the assertion i)(1)(C)  follows from the assertion (i)(2) combined with the
Theorem 1.1 of \cite{ZarhinMF}.

In order to prove the assertion (ii)(b2), notice that we already know (thanks
to the assertion (ii)(b1) ) that $\End(P(C_{f,3}))=\Z[\delta_3]\cong\Z[\zeta_3]$.
This implies that $P(C_{f,3})$ is absolutely simple and has exactly one
principal polarization, which is $\delta_3$-invariant. So, if $P(C_{f,3})$ is
isomorphic to the jacobian of a smooth curve then this isomorphism does respect
the principal polarizations. Now the assertion (ii)(b2) follows from the assertion (i)(1)(C).

%Notice that
%$$3\cdot \mid (2k+1)-(4k+4)\mid =6k+9>(6K+5)+2=m+2=\dim(P(C_{f,3}))+2.$$
%Now the assertion (ii)(b2) follows from the assertion (i)(2) combined with the
%Theorem 1.1 of \cite{ZarhinMF}.
\end{proof}

\end{document}